\documentclass[reqno, 12pt]{amsart}
\usepackage{amsfonts}
\usepackage{amssymb}
\usepackage{bbm}
\usepackage{times}
\usepackage{amsfonts,amssymb,amsmath,amsthm}
\usepackage{enumerate}
\usepackage{bbm}
\usepackage{times}
\usepackage [latin1]{inputenc}
\newtheorem{thm}{Theorem}[section]
\newtheorem{lem}[thm]{Lemma}

\newtheorem{prop}[thm]{Proposition}

\newcounter{other}            
\newtheorem{otherth}[other]{Theorem}              
\newtheorem{otherl}[other]{ Lemma}        

\def\D{\mathbb{D}}
\def\T{\mathbb{T}}
\def\C{\mathbb{C}}
\def\Q{\mathcal{Q}}
\def\B {\mathcal B}
\def \f{\frac}

\begin{document}

\title [Interpolating sequences]
{Interpolating sequences for some subsets of analytic Besov type spaces  }
\author{Ruishen Qian and  Fangqin Ye}
\address{ Ruishen Qian\\
     Lingnan Normal University\\
     Zhanjiang 524048, Guangdong,  China}
\email{qianruishen@sina.cn}

\address{Fangqin Ye\\
     Shantou University\\
    Shantou 515063, Guangdong, China}
\email{fqye@stu.edu.cn}

\thanks{R. Qian was supported by NNSF of China (No. 11801250 and No. 11871257), Overseas Scholarship Program for Elite Young and Middle-aged Teachers of Lingnan Normal University, the Key Subject Program of Lingnan Normal University (No. 1171518004 and No. LZ1905), and the Department of Education of Guangdong Province (No. 2018KTSCX133). F. Ye was supported by NNSF of China (No. 12001352) and  Guangdong Basic and Applied Basic Research Foundation (No. 2019A1515110178).} 
\begin{abstract}
Let $B_p(s)$ be an analytic  Besov type space. Let $M(B_p(s))$ be the class of multipliers of $B_p(s)$ and let $F(p, p-2, s)$ be the M\"obius invariant subspace generated by $B_p(s)$. In this paper, when  $0<s<1$ and $\max\{s, 1-s\}<p\leq 1$,  we  give a completed description of interpolating sequences for  $M(B_p(s))$ and $F(p, p-2, s)\cap H^\infty$. We also  consider  certain condition appeared in this description by   an $L^p$ characterization and the closure of $F(p, p-2, s)$ in the Bloch space.
\end{abstract}

\maketitle

\section{Introduction}

A classical topic in complex analysis is to characterize interpolating sequences for some spaces of  bounded analytic functions. Let $H^\infty$ be the space of bounded analytic functions on the open unit disk $\D$ of the complex plane $\C$. For   $X\subseteq H^\infty$, a sequence $\{z_n\}$ in $\D$ is said to be  an interpolating sequence for $X$ if for each
bounded sequence $\{\zeta_n\}$ of complex values, there exists a function $f\in X$ such that $f(z_n) = \zeta_n$
for every  $n$. A notable  result of L. Carleson \cite{C1} stated   that
$\{z_{n}\}$ is an interpolating sequence for $H^{\infty}$ if and only if $\{z_{n}\}$ is uniformly separated,  namely  there exists a $\delta > 0$ with
$$
\inf_{m}\prod_{n\neq m} \left|\frac{z_m-z_n}{1-\overline{z_m}z_n}\right|\geq \delta.
$$
We refer to \cite{C2, E} for more results related to  this topic.

This paper  is devoted  to consider interpolating sequences for certain  spaces of  bounded analytic functions related to analytic Besov type spaces. Next we recall some definitions and notations.

 For an arc $I$ of the unit circle $\T$, the
 Carleson box  $S(I)$ is
$$
S(I)=\{r\zeta \in \D: 1-|I|<r<1, \ \zeta\in I\},
$$
where $|I|$ is the normalized arclength of $I$.  For $0<s<\infty$,  a nonnegative  Borel measure $\mu$ on $\D$ is said to be an  $s$-Carleson measure  if
$$
\sup_{I\subseteq\T}\frac{\mu(S(I))}{|I|^s}<\infty.
$$
For $s=1$, we get the classical  Carleson measure. Also, if $\mu$ is a $1$-Carleson measure, we write that $\mu$ is a Carleson measure.

Recall that for $a$ and $z$ in $\D$, $\rho(a, z)=|a-z|/|1-\overline{a}z|$ is the pseudo-hyperbolic metric between $a$ and $z$.
A sequence $\{z_n\}_{n=1}^{\infty}$ in $\D$ is called separated if
$$
\inf_{n \not=k}\rho(z_n, z_k)>0.
$$
It is well known (cf. \cite{G}) that $\{z_n\}_{n=1}^{\infty}$ in $\D$ is an interpolating sequence for $H^\infty$ if and only if $\{z_n\}_{n=1}^{\infty}$ is separated  and
$\sum_{n=1}^{\infty} (1-|z_n|^2)\delta_{z_n}$ is a Carleson measure.

Denote by $H(\D)$ the space of analytic functions in $\D$.
For $0<p<\infty$ and $s>1-p$, let $B_p(s)$
be the Besov type space consisting of  functions $f\in H(\D)$ with
$$
\|f\|_{B_p(s)}=|f(0)|+\left(\int_{\D}|f'(z)|^p(1-|z|^2)^{p-2+s}dA(z)\right)^{1/p}<\infty,
$$
where $dA$ is the normalized area measure in $\D$.  It is easy to check that  when  $s\leq 1-p$, the integral above is finite only if $f$ is constant. If $s>1$, then $B_p(s)$ is the Bergman space $A^p_{s-2}$ (see \cite{Zhu}).
Besov type spaces $B_p(s)$ have been studied extensively.
N. Arcozzi, D. Blasi and J. Pau \cite{ABP} characterized universal interpolating sequences for $B_p(s)$ spaces.  D. Blasi and J. Pau \cite{BlasiPau} studied  Hankel-type operators associated with  $B_p(s)$ spaces. See \cite{Co,  Wu} for more results of $B_p(s)$ spaces.

Denote by $M(B_p(s))$ the class of multipliers of the space $B_p(s)$. Namely
$$
M(B_p(s))=\{f\in H(\D): fg\in B_p(s) \ \ \text{for all} \ \ g\in B_p(s)\}.
$$
A nonnegative  Borel measure $\mu$ on $\D$ is called a Carleson measure for  $B_p(s)$ if there is a positive constant $C$ such that
$$
\int_\D |f(z)|^pd\mu(z)\leq C \|f\|_{B_p(s)}^p
$$
for all $f\in B_p(s)$.

Next we recall the M\"obius invariant function space generated by $B_p(s)$. For $a\in \D$, let
$$
\sigma_a(z)=\frac{a-z}{1-\overline{a}z}, \qquad z\in \D.
$$
Namely  $\sigma_a$ is  a M\"obius transformation of $\D$ interchanging the points $0$ and $a$.  Recall that
$$
\text{Aut}(\D)=\{e^{i\theta}\sigma_a:\  \ a\in \D \ \text{and} \ \theta \ \ \text{is real}\}
$$
is the group of M\"obius maps of $\D$. For $0<p<\infty$ and  $0< s<\infty$, the space $F(p, p-2, s)$ consists of those functions $f\in H(\D)$ such that
$$
\|f\|_{F(p, p-2, s)}^p=\sup_{a\in \D}\int_\D |f'(z)|^p(1-|z|^2)^{p-2}(1-|\sigma_a(z)|^2)^s dA(z)<\infty.
$$
Namely, for $f\in F(p, p-2, s)$ if and only if $|f'(z)|^p(1-|z|^2)^{p-2+s}$ is an $s$-Carleson measure.
 It is also known that $F(p, p-2, s)$ contains only constant functions if $p+s\leq 1$. Hence,   to  study
$F(p, p-2, s)$ spaces, we let  $p+s>1$. $F(p, p-2, s)$ spaces and a general family of analytic function spaces  were introduced by R. Zhao \cite{Zhao} and further investigated  in \cite{Rat}.
 If $p=2$ and $s=1$, then $F(p, p-2, s)$ is $BMOA$, the space of analytic functions whose boundary values having bounded mean oscillation on the unit circle.
 If  $s>1$, all nontrivial $F(p, p-2, s)$ spaces are the same and equal to the  Bloch space $\B$ consisting of those functions $f\in H(\D)$ with
 $$
 \sup_{z\in \D}(1-|z|^2)|f'(z)|<\infty.
 $$
 $F(p, p-2, s)$ spaces are M\"obius invariant
in the sense that
$$
\|f\circ  \phi-f(\phi(0))\|_{F(p, p-2, s)}=\|f\|_{F(p, p-2, s)}
$$
for every $f\in F(p, p-2, s)$ and $\phi \in \text{Aut}(\D)$. Indeed, $F(p, p-2, s)$ is equal to the M\"obius invariant subspace generated by $B_p(s)$; that is, $F(p, p-2, s)$ is the space of those functions
$f\in H(\D)$ with
$$
\sup_{\phi\in \text{Aut}(\D)}\left\|f\circ\phi-f(\phi(0))\right\|_{B^p_s}<\infty.
$$
See \cite{AFP} for the general theory of M\"obius invariant function spaces. If $p=2$, then $F(p, p-2, s)$ is the well-known M\"obius invariant  space $\Q_s$ (see \cite{AXZ, Xi1, Xi2}).

As usual, let  $\delta_a$ be the unit point-mass measure at $a\in \D$. N. Arcozzi, D. Blasi and J. Pau \cite{ABP} gave  a series of interesting results on  $B_p(s)$ and  $M(B_p(s))$. In particular, they obtained the following
theorem.
\begin{otherth}\label{ABP}
Suppose  $1<p<\infty$, $0<s<1$ and $\{z_n\}_{n=1}^{\infty}$ is a sequence in $\D$. Then the following conditions are equivalent:
\begin{enumerate}
  \item [(a)]$\{z_n\}_{n=1}^{\infty}$ is an interpolating sequence for $M(B_p(s))$;
  \item [(b)]$\{z_n\}_{n=1}^{\infty}$ is a separated sequence  and $\sum_{n=1}^{\infty} (1-|z_n|^2)^s\delta_{z_n}$ is a Carleson measure for $B_p(s)$.
  \end{enumerate}
\end{otherth}

Ch. Yuan and C. Tong \cite{YT} characterized interpolating sequences for $F(p, p-2, s)\cap H^{\infty}$ which generalizes the case of $\Q_s\cap H^{\infty}$ in \cite{NX}. We recall Ch. Yuan and C. Tong's result as follows.

\begin{otherth}\label{YT}
Suppose  $1<p<\infty$, $0<s<1$ and $\{z_n\}_{n=1}^{\infty}$ is a sequence in $\D$.  Then the following conditions are equivalent:
\begin{enumerate}
  \item [(a)]$\{z_n\}_{n=1}^{\infty}$ is an interpolating sequence for $F(p, p-2, s)\cap H^{\infty}$;
  \item [(b)]$\{z_n\}_{n=1}^{\infty}$ is a separated sequence  and $\sum_{n=1}^{\infty} (1-|z_n|^2)^s\delta_{z_n}$ is an $s$-Carleson measure.
  \end{enumerate}
\end{otherth}

In light of   Theorems \ref{ABP} and \ref{YT},  it is quite  natural to characterize     interpolating sequences for $M(B_p(s))$ and $F(p, p-2, s)\cap H^{\infty}$ when $0<s<1$ and $1-s<p \leq 1$.
In this paper, we get the corresponding result for  $0<s<1$ and  $\max\{s, 1-s\}<p\leq 1$. In fact, we observe that  $F(p, p-2, s)\cap H^{\infty}=M(B_p(s))$ when $0<p\leq 1$ and  $1-p<s\leq 1$.
For $0<s<1$ and certain $\{z_n\}_{n=1}^{\infty}$ in $\D$, $s$-Carleson measures $\sum_{n=1}^{\infty} (1-|z_n|^2)^s\delta_{z_n}$  play an important role in our interpolating theorem. This inspires us to study them further.  In Section 3, we consider these $s$-Carleson measures via an $L^p$ description. In Section 4, we apply these $s$-Carleson measures to investigate proper  inclusion relation related to the closure of $F(p, p-2, s)$  in the Bloch space.

Throughout this paper,  we  write $a\lesssim b$ if there exists a positive constant $C$ such that $a\leq Cb$.
 The symbol $a\thickapprox b$ means that $a\lesssim b\lesssim a$.

\section{Interpolating sequence for $M(B_p(s))$ and $F(p, p-2, s)\cap H^{\infty}$  with small exponents }

Suppose $0<s<1$ and  $ 1-s<p\leq 1$. The function $\log(1-z) \in F(p, p-2, s)\backslash H^\infty$. From \cite[Proposition A]{PR},  there exists a function in $H^\infty\backslash F(p, p-2, s)$.
Hence $F(p, p-2, s)\nsubseteq H^\infty$ and $H^\infty \nsubseteq F(p, p-2, s)$.

The main result in this section  is the following theorem.

\begin{thm} \label{mainresult}
Let $0<s<1$, $\max\{s, 1-s\}<p\leq 1$ and $t>0$. Suppose $\{z_n\}_{n=1}^{\infty}$ is a sequence in $\D$. Then the following conditions are equivalent:
\begin{enumerate}
\item [(a)] $\{z_n\}_{n=1}^{\infty}$ is an interpolating sequence for $M(B_p(s))$;
  \item [(b)]$\{z_n\}_{n=1}^{\infty}$ is an interpolating sequence for $F(p, p-2, s)\cap H^{\infty}$;
  \item [(c)]$\{z_n\}_{n=1}^{\infty}$ is a separated sequence  and $\sum_{n=1}^{\infty} (1-|z_n|^2)^s\delta_{z_n}$ is an $s$-Carleson measure;
  \item [(d)] $\{z_n\}_{n=1}^{\infty}$ is a separated sequence  and
  $$
  \sup_{a\in \D} \sum_{n=1}^{\infty} \frac{(1-|a|^2)^t(1-|z_n|^2)^s }{|1-\overline{a}z_n|^{s+t}}<\infty.
  $$
  \end{enumerate}
\end{thm}

A function  $f\in H(\D)$  is called an inner function if it is an $H^\infty$-function with radial limits of modulus one almost everywhere on  $\T$.
A sequence $\{a_k\}_{k=1}^\infty$  in $\D$ is said to be a Blaschke sequence if
$$
\sum_{k=1}^\infty (1-|a_k|)<\infty.
$$
This condition  ensures  the convergence of the corresponding Blaschke product $B$ which is
$$
B(z)=\prod_{k=1}^\infty \f{|a_k|}{a_k}\f{a_k-z}{1-\overline{a_k}z}.
$$
Every Blaschke product is an  inner function. A Blaschke product is said to be an interpolating Blaschke product if its  zero sequence is  an interpolating sequence for $H^\infty$.

F. P\'erez-Gonz\'alez and J. R\"atty\"a \cite{PR} described  inner functions in $F(p, p-2, s)$ spaces as follows.

\begin{otherth} \label{F-Inner}
Let $0<s<1$ and $p>\max\{s, 1-s\}$. Then an inner function belongs to $F(p, p-2, s)$ if and only if it is a Blaschke product associated with a sequence $\{z_k\}_{k=1}^\infty\subseteq \D$ which satisfies that
$\sum_k (1-|z_k|)^s \delta_{z_k}$ is an $s$-Carleson measure;  that is
$$
\sup_{a\in \D}\sum_{k=1}^\infty  \left(1-|\sigma_a(z_k)|^2\right)^s<\infty.
$$
\end{otherth}

We also  recall the following well-known description of Carleson type measures  (cf. \cite{Bla} ).
\begin{otherl}\label{S-CM}
Suppose $s>0$, $t>0$ and $\mu$ is a nonnegative Borel measure on $\D$. Then $\mu$ is an $s$-Carleson measure if and only if
$$
\sup_{z\in \D}\int_{\D} \frac{(1-|z|^2)^t}{|1-\overline{z}w|^{s+t}}d\mu(w)<\infty.
$$
\end{otherl}

The following (a) and (b) of Theorem \ref{carlesonBps}  are from  Theorem 3.3 and Theorem 4.2 in  \cite{Wu},  respectively.

\begin{otherth} \label{carlesonBps}
The following statements are true:
\begin{enumerate}
  \item [(a)] let  $0<p<\infty$, $0\leq s<\infty$ and $p+s>1$. Then $g\in M(B_p(s))$ if and only if $g\in H^{\infty}$ and $|g'(z)|^p(1-|z|^2)^{p-2+s}dA(z)$ is a Carleson measure for $B_p(s)$;
  \item [(b)] let  $0<p\leq 1$ and  $1-p<s\leq 1$. Then a nonnegative  Borel measure $\mu$ on $\D$ is  a Carleson measure for  $B_p(s)$ if and only if
  $$
  \int_\D \f{d\mu(w)}{|1-\overline{w}z|^2}=O((1-|z|^2)^{s-2}).
  $$
\end{enumerate}
\end{otherth}

When $0<p<\infty$ and $s>1$, recall that $F(p, p-2, s)$ is the Bloch space $\B$, and $B_p(s)$ is the Bergman space $A^p_{s-2}$ (cf. \cite{Zhu}). Hence, for $0<p<\infty$ and $s>1$,
$M(B_p(s))=F(p, p-2, s)\cap H^{\infty}=H^{\infty}$.

From \cite[Theorem 5.2]{ALXZ}, if $0<s\leq 1$, then $M(B_2(s))\subseteq \Q_s$.
By  \cite[Lemma 5.1]{ABP}, if $1<p\leq 2$ and $0<s<1$, then  $M(B_p(s))\subseteq \Q_s\cap H^{\infty}$.

Suppose  $1<p\leq 2$ and $0<s<1$.   Proposition \ref{M(Bps) F} below gives $M(B_p(s))\subseteq F(p, p-2, s)\cap H^{\infty}$. It is known that (cf. \cite[Remark 2]{Ye}) $F(p, p-2, s)\cap H^{\infty}\subseteq \Q_s\cap H^{\infty}$.
Thus we also get that  $M(B_p(s))\subseteq \Q_s\cap H^{\infty}$ when $1<p\leq 2$ and $0<s<1$.

\begin{prop} \label{M(Bps) F}
The following  statements hold:
\begin{enumerate}
  \item [(a)] if  $0<p\leq 1$ and  $1-p<s\leq 1$, then $$M(B_p(s))= F(p, p-2, s)\cap H^{\infty}; $$
  \item [(b)] if $p>1$ and $0<s\leq 1$, then $$M(B_p(s))\subseteq F(p, p-2, s)\cap H^{\infty}. $$
\end{enumerate}
 \end{prop}
\begin{proof} (a)\
 From  Theorem \ref{carlesonBps}, $g\in M(B_p(s))$ if and only if $g\in H^{\infty}$  and
\begin{equation}\label{Fconditon}
  \int_\D \f{|g'(w)|^p(1-|w|^2)^{p-2+s}dA(w)}{|1-\overline{w}z|^2}=O((1-|z|^2)^{s-2}).
\end{equation}
Taking  $t=2-s$ and $d\mu(w)=|g'(w)|^p(1-|w|^2)^{p-2+s}dA(w)$ in Lemma \ref{S-CM}, we see that (\ref{Fconditon}) holds if and only if
$|g'(w)|^p(1-|w|^2)^{p-2+s}dA(w)$ is an $s$-Carleson measure. In other words, (\ref{Fconditon}) holds if and only if $g\in F(p, p-2, s)$.
Hence $M(B_p(s))= F(p, p-2, s)\cap H^{\infty}$.

(b)\ Let $g\in M(B_p(s))$. For $a\in \D$,  set
$$
f_a(z)= \frac{(1-|a|^2)^{s/p}}{(1-\overline{a}z)^{\frac{2s}{p}}}, \ \ z\in \D.
$$
Note that $p>1$ and $0<s\leq 1$. We see that (cf. \cite[p. 407]{BlasiPau})
$$
\sup_{a\in \D} \|f_a\|_{B_p(s)}<\infty.
$$
Bear in mind that $|g'(z)|^p(1-|z|^2)^{p-2+s}dA(z)$ is a Carleson measure for $B_p(s)$. Then
\begin{eqnarray*}
\int_\D  |f_a(z)|^p |g'(z)|^p(1-|z|^2)^{p-2+s}dA(z) \lesssim \|f_a\|_{B_p(s)}^p
\end{eqnarray*}
for all $a\in \D$. This gives
$$
\sup_{a\in \D} \int_\D  \frac{(1-|a|^2)^{s}}{|1-\overline{a}z|^{2s}} |g'(z)|^p(1-|z|^2)^{p-2+s}dA(z)<\infty.
$$
Thus $g\in F(p, p-2, s)$.  Hence $M(B_p(s))\subseteq F(p, p-2, s)\cap H^{\infty}$.
\end{proof}

\vspace{0.1truecm}

\noindent{\bf Proof of Theorem \ref{mainresult}.}
The equivalence of (a) and (b) is from Proposition  \ref{M(Bps) F}. By Lemma \ref{S-CM},  (c) and (d) are also equivalent.

(b)$\Rightarrow$(c). Note that $0<s<1$ and $\max\{s, 1-s\}<p\leq 1$. By a simple observation (cf. \cite[p. 207]{BP}), we see that $F(p, p-2, s)\subseteq F(2, 0, s)$.
Since $\{z_n\}_{n=1}^{\infty}$ is an interpolating sequence for $F(p, p-2, s)\cap H^{\infty}$, $\{z_n\}_{n=1}^{\infty}$ is also an interpolating sequence for $F(2, 0, s) \cap H^{\infty}$.
By Theorem \ref{YT}, we get that (c) holds.

(c)$\Rightarrow$(b). Suppose $\{z_n\}_{n=1}^{\infty}$ is a separated sequence  and $\sum_{n=1}^{\infty} (1-|z_n|^2)^s\delta_{z_n}$ is an $s$-Carleson measure. Then
$$
\sum_{z_n \in S(I)} (1-|z_n|^2)\leq |I|^{1-s}\sum_{z_n \in S(I)}(1-|z_n|^2)^s\lesssim |I|
$$
for all $I\subseteq \T$. Hence $\sum_{n=1}^{\infty} (1-|z_n|^2)\delta_{z_n}$ is a Carleson measure. Thus $\{z_n\}_{n=1}^{\infty}$ is an interpolating sequence for $H^{\infty}$.
Then
$$
\inf_k \prod_{j\neq k} \rho(z_j, z_k) =:\delta>0.
$$
Let $\{a_n\}_{n=1}^{\infty}$ be any bounded sequence of complex numbers. From  J. Earl's constructive proof of L. Carleson's interpolating theorem
for $H^\infty$ (see \cite{E}), there exists  a complex  constant $C$ and a Blaschke product $B$ with zeros $\{\zeta_n\}_{n=1}^{\infty}$ such that
$$f(z)=:C\sup_n\{a_n\}B(z)$$ satisfies
$$
f(z_n)=a_n,\ \ n=1, 2,\cdots,
$$
and
\begin{equation}\label{zn zetan}
\rho(z_n, \zeta_n)\leq \frac{\delta}{3}, \  n=1, 2, \cdots.
\end{equation}
It suffices to prove that $B\in F(p, p-2, s)$.  Because of (2), we know from \cite[p. 69]{Zhu} and \cite[Lemma 4.30]{Zhu}  that
\begin{equation}\label{compare1}
1-|z_n|\thickapprox1-|\zeta_n|\thickapprox |1-\overline{z_n} \zeta_n|
\end{equation}
for all positive integers $n$, and
\begin{equation}\label{compare2}
 |1-\overline{z_n} w|\thickapprox  |1-\overline{\zeta_n} w|
\end{equation}
for all positive integers $n$ and all $w\in \D$. Also, the equivalence of (c) and (d) yields
\begin{equation}\label{condition d}
  \sup_{a\in \D} \sum_{n=1}^{\infty} \frac{(1-|a|^2)^s(1-|z_n|^2)^s }{|1-\overline{a}z_n|^{2s}}<\infty.
 \end{equation}
Joining (\ref{compare1}), (\ref{compare2}) and (\ref{condition d}), we get
 $$
  \sup_{a\in \D} \sum_{n=1}^{\infty} \frac{(1-|a|^2)^s(1-|\zeta_n|^2)^s }{|1-\overline{a}\zeta_n|^{2s}}<\infty.
$$
Combining this with Theorem \ref{F-Inner},  we obtain $B\in F(p, p-2, s)$. This finishes the proof.
\hfill{$\square$}

For $0<s\leq 1/2$,  the assumption   $\max\{s, 1-s\}<p\leq 1$ in Theorem \ref{mainresult} is sharp because  $\max\{s, 1-s\}=1-s$ for such  case, and the condition $ 1-s<p\leq 1$
is  necessary  because of  the non-triviality  of $B^p_s$ and $F(p, p-2, s)$. When $1/2<s<1$ and $1-s<p\leq s$, we do not have a good answer for the characterization of  interpolating sequence for $M(B_p(s))$ and
$F(p, p-2, s)\cap H^{\infty}$.

\section{An $L^p$ description related to  $s$-Carleson measure $\sum_{n=1}^{\infty} (1-|z_n|^2)^s\delta_{z_n}$ with $0<s<1$}

Suppose $\{z_n\}_{n=1}^{\infty}$ is a sequence in $\D$. For  $0<s<1$, the $s$-Carleson measure $\sum_{n=1}^{\infty} (1-|z_n|^2)^s\delta_{z_n}$   plays a key role in our Theorem \ref{mainresult}, which characterizes interpolating sequence for $M(B_p(s))$ and $F(p, p-2, s)\cap H^{\infty}$  for  certain range of parameters $p$ and $s$.  In this section, we consider  further this kind of  $s$-Carleson measures. In fact, these $s$-Carleson measures defined by sequences $\{z_n\}_{n=1}^{\infty}$ were  used to characterize inner functions in $\Q_s$ and $F(p, p-2, s)$ spaces (see\cite{EX, PR}), and also used to study solutions of second order complex differential equation having certain pre-given zeros (cf. \cite{Gr, Ye}). We refer to \cite{NX, Qian} for more results associated with these $s$-Carleson measures.

Now we  recall  the classical case of $s=1$. Namely we let  $\{z_n\}_{n=1}^{\infty}$ be a sequence  in $\D$ such that $\sum_{n=1}^{\infty} (1-|z_n|^2)\delta_{z_n}$ is a Carleson measure.  From  \cite{MS}, this is equivalent to say that   $\{z_n\}_{n=1}^{\infty}$ is a  union of finitely many  uniformly separated sequences in $\D$. By an  $L^p$ description,  C. Nolder \cite[Theorem D]{No} characterized the union of finitely many  uniformly separated sequences. We refer to P. Duren and A. Schuster \cite{DS} and  A. Nicolau \cite{Ni} for more interesting results on these sequences.

Give   an $s$-Carleson measure $\sum_{n=1}^{\infty} (1-|z_n|^2)^s\delta_{z_n}$, where  $0<s\leq 1$ and $\{z_n\}_{n=1}^{\infty}$ is certain sequence in $\D$. Then it is easy to see that $\{z_n\}_{n=1}^{\infty}$ is also a Blaschke sequence.
For $s=1$,  C. Nolder  \cite[Theorem D]{No} obtained an $L^p$ characterization related to  these $s$-Carleson measures as follows.

\begin{otherth}\label{Nolder}
Suppose $0<p<2$ and  $\{z_n\}_{n=1}^{\infty}$  is a Blaschke sequence in $\D$. Let $B$ be the Blaschke product whose zero sequence is  $\{z_n\}_{n=1}^{\infty}$.    Then the following conditions are equivalent:
\begin{enumerate}
  \item [(a)] $\sum_{n=1}^{\infty} (1-|z_n|^2)\delta_{z_n}$ is a Carleson measure;
  \item [(b)]
  $$
  \sup_{\phi\in \text{Aut}(\D) } \int_\D \frac{1}{|B(\phi(z))|^p}dA(z)<\infty.
  $$
\end{enumerate}
\end{otherth}

In this section, motivated by   C. Nolder's result above,   we obtain an   $L^p$ description  associated with  these $s$-Carleson measures with  $0<s<1$. As shown by Theorem \ref{2main}  below, the case of $0<s<1$ seems to be more delicate.

Next we recall some well-known estimates (cf. \cite[Lemma 3.10]{Zhu}).

\begin{otherl}\label{usefulestimates}
Suppose $z\in \D$, $c$ is real and  $t>-1$. Then
$$
\int_\D\frac{(1-|w|^2)^t}{|1-\bar{z}w|^{2+t+c}}dA(w)\thickapprox
\begin{cases}1 & \enspace \text{if} \ \ \ c<0,\\
                     \log\frac{2}{1-|z|^2} & \enspace  \text{if} \ \  \ c=0,\\
                     \frac{1}{(1-|z|^2)^c} & \enspace \text{if}\ \  \ c>0,
                   \end{cases}
$$
as $|z|\rightarrow 1^{-}$.
\end{otherl}

Now we give our result in this section.
\begin{thm}\label{2main}
Suppose $0<s<1$, $\max\{s, 1-s\}<p\leq 1$ and $\{z_n\}_{n=1}^{\infty}$  is a Blaschke sequence in $\D$. Let $B$ be the Blaschke product  whose zero sequence is  $\{z_n\}_{n=1}^{\infty}$.    Then the following conditions are equivalent:
\begin{enumerate}
\item [(a)]  $\sum_{n=1}^{\infty} (1-|z_n|^2)^s\delta_{z_n}$   is an $s$-Carleson measure;
\item [(b)] $$
  \sup_{\phi\in \text{Aut}(\D) } \int_\D \left|\frac{B'(\phi(z))}{B(\phi(z))}\right|^p (1-|\phi(z)|^2)^{p} (1-|z|^2)^{s-2} dA(z)<\infty.
  $$
\end{enumerate}
\end{thm}
\begin{proof} By a change of variables, (b) is equivalent to
\begin{equation}\label{change vari}
\sup_{a\in \D} \int_\D \left|\frac{B'(z)}{B(z)}\right|^p (1-|z|^2)^{p-2}(1-|\sigma_a(z)|^2)^s dA(z)<\infty.
\end{equation}
For the sake of comparison with Theorem \ref{Nolder}, we use (b) in the expressions  of this theorem.

First, let (b) hold. Then (\ref{change vari}) is true. Combining this with the fact that $|B(z)|\leq 1$ for all $z\in \D$, we see that
$$
\sup_{a\in \D} \int_\D |B'(z)|^p (1-|z|^2)^{p-2}(1-|\sigma_a(z)|^2)^s dA(z)<\infty,
$$
which yields $B\in F(p, p-2, s)$. Note that  $0<s<1$ and $\max\{s, 1-s\}<p\leq 1$. By Theorem \ref{F-Inner}, $\sum_{n=1}^{\infty} (1-|z_n|^2)^s\delta_{z_n}$   is an $s$-Carleson measure.

On the other hand, let $\sum_{n=1}^{\infty} (1-|z_n|^2)^s\delta_{z_n}$ be an $s$-Carleson measure. By Lemma \ref{S-CM}, we get
\begin{equation}\label{scm inter}
\sup_{a\in \D} \sum_{n=1}^{\infty} \frac{(1-|a|^2)^s  (1-|z_n|^2)^s}{|1-\overline{a}z_n|^{2s}}<\infty.
\end{equation}
Note that
$$
\frac{B'(z)}{B(z)}=\sum_{n=1}^{\infty}\frac{|z_n|^2-1}{(z_n-z)(1-\overline{z_n}z)}
$$
for $z\in \D\setminus \{z_n\}_{n=1}^{\infty}$.
Since $0<\max\{s, 1-s\}<p\leq 1$, we have
$$
\left|\frac{B'(z)}{B(z)}\right|^p \leq \sum_{n=1}^{\infty}\frac{(1-|z_n|^2)^p}{|z_n-z|^p |1-\overline{z_n}z|^p}.
$$
 Hence, for each  $a\in \D$, we deduce that
\begin{eqnarray}
&~& \int_\D \left|\frac{B'(z)}{B(z)}\right|^p (1-|z|^2)^{p-2}(1-|\sigma_a(z)|^2)^s dA(z) \nonumber \\
&\leq& \sum_{n=1}^{\infty}  (1-|z_n|^2)^p \int_\D  \frac{(1-|z|^2)^{p-2}(1-|\sigma_a(z)|^2)^s}{|z_n-z|^p |1-\overline{z_n}z|^p} dA(z) \nonumber \\
&=& (1-|a|^2)^s  \sum_{n=1}^{\infty}    \int_\D \frac{(1-|z_n|^2)^p(1-|z|^2)^{p-2+s}}{|1-\overline{a}z|^{2s}|1-\overline{z_n}z|^{2p}|\sigma_{z_n}(z)|^p} dA(z). \label{mid 1}
\end{eqnarray}
Making a change of variables $w=\sigma_{z_n}(z)$ in the integral in (\ref{mid 1}), we get
\begin{eqnarray}
&~&\int_\D \frac{(1-|z|^2)^{p-2+s}}{|1-\overline{a}z|^{2s}|1-\overline{z_n}z|^{2p}|\sigma_{z_n}(z)|^p} dA(z) \nonumber \\
&=&  \int_\D  \frac{(1-|\sigma_{z_n}(w)|^2)^{p+s}}{|w|^p (1-|w|^2)^2 |1-\overline{a}\sigma_{z_n}(w)|^{2s}|1-\overline{z_n}\sigma_{z_n}(w)|^{2p}} dA(w) \nonumber \\
&=&   (1-|z_n|^2)^{s-p} \int_\D \frac{(1-|w|^2)^{p+s-2}}{|w|^p |1-\overline{z_n} w|^{2s}|1-\overline{a}\sigma_{z_n}(w)|^{2s}}   dA(w). \label{mid 2}
\end{eqnarray}
For $0<p\leq 1$, $dA(w)/|w|^p=r^{1-p}drd\theta$, where $w=re^{i\theta}$. Thus the integral  in (\ref{mid 2}) is convergent. In fact, from \cite[Lemma 4.26]{Zhu}, we know
\begin{eqnarray}
&~&\int_\D \frac{(1-|w|^2)^{p+s-2}}{|w|^p |1-\overline{z_n} w|^{2s}|1-\overline{a}\sigma_{z_n}(w)|^{2s}}   dA(w) \nonumber \\
&\lesssim&  \int_\D \frac{(1-|w|^2)^{p+s-2}}{ |1-\overline{z_n} w|^{2s}|1-\overline{a}\sigma_{z_n}(w)|^{2s}}   dA(w) \label{mid 3}
\end{eqnarray}
for all $z_n$ and $a\in \D$.  We compute that
\begin{eqnarray*}
 1-\overline{a}\sigma_{z_n}(w)&=& \frac{1-\overline{a}z_n-w(\overline{z_n}-\overline{a})}{1-\overline{z_n}w}\\
 &=&(1-\overline{a}z_n)\frac{1-w \overline{\sigma_{z_n}(a)}}{1-\overline{z_n}w}.
\end{eqnarray*}
Combining this  with  Lemma \ref{usefulestimates}, we obtain
\begin{eqnarray}
&~&\int_\D \frac{(1-|w|^2)^{p+s-2}}{ |1-\overline{z_n} w|^{2s}|1-\overline{a}\sigma_{z_n}(w)|^{2s}}   dA(w) \nonumber \\
&=& \frac{1}{|1-\overline{a}z_n|^{2s}} \int_\D \frac{(1-|w|^2)^{p+s-2}}{|1-w \overline{\sigma_{z_n}(a)}|^{2s}} dA(w) \nonumber \\
&\thickapprox& \frac{1}{|1-\overline{a}z_n|^{2s}},  \label{mid 4}
\end{eqnarray}
where we use   $0<\max\{s, 1-s\}<p$.

Consequently, formulas (\ref{scm inter}),  (\ref{mid 1}), (\ref{mid 2}),  (\ref{mid 3}) and (\ref{mid 4}) yield
\begin{eqnarray*}
&~& \sup_{a\in \D}\int_\D \left|\frac{B'(z)}{B(z)}\right|^p (1-|z|^2)^{p-2}(1-|\sigma_a(z)|^2)^s dA(z)  \\
 &\lesssim& \sup_{a\in \D}  \sum_{n=1}^{\infty} \frac{(1-|a|^2)^s  (1-|z_n|^2)^s}{|1-\overline{a}z_n|^{2s}}<\infty.
\end{eqnarray*}
The proof is complete.
\end{proof}

\section{An application of $s$-Carleson measures $\sum_{n=1}^{\infty} (1-|z_n|^2)^s\delta_{z_n}$ to $C_\B(F(p, p-2, s))$}

In 1974,  J. Anderson, J.  Clunie and Ch.  Pommerenke \cite{ACP} posed the following  problem:  what is the closure  in the Bloch norm of $H^\infty$? This problem remains open.
For a space $X$ of functions analytic in $\D$, let $C_\B(X)$ be the closure of $X$ in the Bloch norm.  An important result in this field is P. Jones' characterization of $C_B(BMOA)$ (see \cite{GZ}).
Later, R. Zhao \cite{Zhao1} described $C_\B(F(p, p-2, s))$ when $0<s\leq 1$ and $1\leq p<\infty$. R. Zhao's characterization of $C_\B(F(p, p-2, s))$ is independent of the parameter $p$. In this section, we consider the proper inclusion relation related to $C_\B(F(p, p-2, s))$ via interpolating Blaschke products associated with $s$-Carleson measures $\sum_{n=1}^{\infty} (1-|z_n|^2)^s\delta_{z_n}$.

For $f\in H(\D)$ and $\epsilon>0$, let $\Omega_\epsilon(f)=\{z\in \D:  |f'(z)|(1-|z|^2)\geq \epsilon\}$. We recall R. Zhao's characterization of $C_\B(F(p, p-2, s))$ as follows.
\begin{otherth} \label{zhao}
Let $0<s\leq 1$,  $1\leq p <\infty$ and $0\leq t<\infty$. Suppose  $f\in \B$. Then $f\in C_\B (F(p, p-2, s))$ if and only if for every $\epsilon>0$,
$$
\sup_{a\in \D} \int_{\Omega_\epsilon(f)} |f'(z)|^t (1-|z|^2)^{t-2} (1-|\sigma_a(z)|^2)^s dA(z)<\infty.
$$
\end{otherth}

Based on some arguments of $s$-Carleson measures $\sum_{n=1}^{\infty} (1-|z_n|^2)^s\delta_{z_n}$,  we obtain the following proper inclusion relation, which is the main result in this section.

\begin{thm} \label{3main}
Let $0<s\leq 1$ and $1\leq p <\infty$. Then
$$
C_\B(F(p, p-2, s)) \subsetneqq \bigcap_{s<t<\infty} C_\B(F(p, p-2, t)).
$$
\end{thm}

Before the proof of Theorem \ref{3main}, we need some preliminary results. The following useful   inequality is well known (cf.  \cite[Lemma 2.5]{O}).

\begin{otherl} \label{inequ}
Suppose   $s>-1$, $r>0$, $t>0$, and $t<s+2<r$. Then there exists a positive constant $C$ such that
$$
\int_\D\frac{(1-|w|^2)^s}{|1-\overline{w}z|^r |1-\overline{w}\zeta|^t}dA(w)\leq C \frac{(1-|z|^2)^{2+s-r}}{|1-\overline{\zeta}z|^t}
$$
for all $z$, $\zeta\in \D$.
\end{otherl}

We also need the following result, which characterizes interpolating Blaschke products in $C_\B(F(p, p-2, s))$.

\begin{lem}\label{IB}
Suppose   $0<s<1$, $1\leq p<\infty$,  and   $I$ is  an interpolating Blaschke product with zero set $\{z_n\}_{n=1}^{\infty}$. Then $I\in C_\B(F(p, p-2, s))$ if and only if
$\sum_{n=1}^{\infty} (1-|z_n|^2)^s\delta_{z_n}$ is an $s$-Carleson measure.
\end{lem}
\begin{proof}
First, let $I\in C_\B(F(p, p-2, s))$.  We adapt an argument from \cite[p. 1220]{GMP} (see also \cite{BG}). Since $I$ is an interpolating Blaschke product, there exists $\gamma>0$ such that
$$
\inf_n(1-|z_n|^2)|I'(z_n)|\geq \gamma.
$$
Take $\epsilon\in (0, \f{\gamma}{4})$. From   \cite[p. 681]{GPV}),  there exists a positive constant $\rho$ such that
$$
(1-|z|^2)|I'(z)|\geq \epsilon
$$
on every  $D_h(z_n, \rho)$, where
$D_h(z_n, \rho)$ is the hyperbolic disk with center $z_n$ and radius $\rho$. Namely,
$$
D_h(z_n, \rho)=\{z\in \D: \f{1}{2}\log\f{1+|\sigma_{z_n}(z)|}{1-|\sigma_{z_n}(z)|}<\rho\}.
$$
Then
$$
\bigcup_{n=1}^\infty D_h(z_n, \rho) \subseteq \{z\in \D: (1-|z|^2)|I'(z)|\geq \epsilon\}.
$$
Since $\{a_{n}\}$ is uniformly separated, we can select $\rho$ such that these  hyperbolic disks are pairwise disjoints. Combining these with
Theorem \ref{zhao}, we see  that
\begin{eqnarray*}
\infty&>& \sup_{w\in\D} \int_{\{z\in \D:\  (1-|z|^2)|I'(z)|\geq \epsilon\}} \frac { (1-|\sigma_w(z)|^2)^s}{(1-|z|^2)^2} dA(z)\\
&\geq&  \sup_{w\in\D} \sum_{n=1}^\infty \int_{D_h(z_n, \rho)}\frac { (1-|\sigma_w(z)|^2)^s}{(1-|z|^2)^2} dA(z).
\end{eqnarray*}
By \cite[Proposition 4.5]{Zhu} and \cite[Lemma 4.30]{Zhu}, we know that
$$
1-|z|\thickapprox 1-|z_n|\thickapprox |1-\overline{z_n}z|, \ \ |1-\overline{w}z|\thickapprox |1-\overline{w}z_n|,
$$
for all $z\in D_h(z_n, \rho)$ and all  $w\in \D$.
Therefore,
$$
\sup_{w\in\D}\sum_{n=1}^{\infty}(1-|\sigma_w(z_n)|^2)^s<\infty,
$$
which gives that $\sum_{n=1}^{\infty} (1-|z_n|^2)^s\delta_{z_n}$ is an $s$-Carleson measure.

On the other hand, let $\sum_{n=1}^{\infty} (1-|z_n|^2)^s\delta_{z_n}$ be an $s$-Carleson measure. Bear in mind  that
$$
|I'(z)|\leqslant \sum_{n=1}^{\infty}\frac{1-|z_{n}|^{2}}{|1-\overline{z_{n}}z|^{2}}, \ \ z\in \D.
$$
Then for every  $\epsilon>0$, one gets that
\begin{eqnarray*}
&~& \sup_{a\in \D} \int_{\Omega_\epsilon(I)} \frac {(1-|\sigma_a(z)|^2)^s }{(1-|z|^2)^2}dA(z)\\
&\leq& \frac{1}{\epsilon}\sup_{a\in \D} \int_{\Omega_\epsilon(I)} (1-|z|^2)|I'(z)| \frac {(1-|\sigma_a(z)|^2)^s }{(1-|z|^2)^2} dA(z)\\
&\leq&  \frac{1}{\epsilon}\sup_{a\in \D} \sum_{n=1}^{\infty} \int_{\Omega_\epsilon(I)} \frac{1-|z_{n}|^{2}}{|1-\overline{z_{n}}z|^{2}}
\frac {(1-|\sigma_a(z)|^2)^s }{1-|z|^2} dA(z)\\
&\leq & \sup_{a\in \D} \sum_{n=1}^{\infty} \frac {(1-|a|^2)^s (1-|z_n|^2)}{\epsilon}  \int_\D
\frac{(1-|z|^2)^{s-1}}{|1-\overline{z_n}z|^2|1-\overline{a}z|^{2s}} dA(z).
\end{eqnarray*}
Note that $0<s<1$. By Lemma \ref{inequ}, we see that
$$
\int_\D
\frac{(1-|z|^2)^{s-1}}{|1-\overline{z_n}z|^2|1-\overline{a}z|^{2s}} dA(z)\lesssim \frac{(1-|z_n|^2)^{s-1}}{|1-\overline{z_{n}}a|^{2s}}.
$$
Thus,
$$
\sup_{a\in \D} \int_{\Omega_\epsilon(I)} \frac {(1-|\sigma_a(z)|^2)^s }{(1-|z|^2)^2}dA(z)\lesssim
\sup_{a\in\D}\sum_{n=1}^{\infty}(1-|\sigma_a(z_n)|^2)^s<\infty.
$$
By Theorem \ref{zhao}, one gets $I\in C_\B(F(p, p-2, s))$.
The proof is complete.
\end{proof}

\vspace{0.1truecm}
\noindent{\bf Proof of Theorem \ref{3main}.}
Note that for $1\leq p<\infty$ and $0<s_1\leq s_2<\infty$, $F(p, p-2, s_1)\subseteq F(p, p-2, s_2)$. Hence $C_\B(F(p, p-2, s_1))\subseteq C_\B(F(p, p-2, s_2))$. Thus for  $0<s\leq 1$ and $1\leq p <\infty$, we have
$$
C_\B(F(p, p-2, s)) \subseteq  \bigcap_{s<t<\infty} C_\B(F(p, p-2, t)).
$$

If $s=1$, then  $C_\B(F(p, p-2, s))=C_\B(BMOA)$ and
$$
\bigcap_{s<t<\infty} C_\B(F(p, p-2, t))=\B.
$$
It is known from \cite{GZ} that $C_\B(BMOA)\subsetneqq \B$. Thus
$$
C_\B(F(p, p-2, 1)) \subsetneqq \bigcap_{1<t<\infty} C_\B(F(p, p-2, t)).
$$

Let $0<s<1$. From Lemma 2.1 in \cite{BWY}, there exists a Blaschke sequence $\{z_n\}_{n=1}^{\infty}$ such that $\sum_{n=1}^{\infty} (1-|z_n|^2)^s\delta_{z_n}$ is not an $s$-Carleson measure and
\begin{equation}\label{3-mid1}
\sup_{w\in \D} \sum_{n=1}^{\infty} \left(1-|\sigma_w(z_n)|^2\right)^{s}\left(\f{1}{\ln \f{2}{1-|\sigma_w(z_n)|^2}}\right)^2<\infty.
\end{equation}
Denote by $I$ the Blaschke product whose zero sequence is   $\{z_n\}_{n=1}^{\infty}$. Formula (\ref{3-mid1}) implies that
$$
\sup_{w\in \D} \sum_{n=1}^{\infty} \left(1-|\sigma_w(z_n)|^2\right)<\infty;
$$
that is,  $\sum_{n=1}^{\infty} (1-|z_n|^2)\delta_{z_n}$ is a Carleson measure. Hence $I$ is an interpolating Blaschke product. By Lemma \ref{IB},
\begin{equation}\label{3-mid2}
I\not\in C_\B(F(p, p-2, s)).
\end{equation}
If $0<s<t<1$, it is clear that
\begin{eqnarray*}
&~&\sup_{w\in \D} \sum_{n=1}^{\infty} \left(1-|\sigma_w(z_n)|^2\right)^t\\
&\lesssim&  \sup_{w\in \D} \sum_{n=1}^{\infty} \left(1-|\sigma_w(z_n)|^2\right)^{s}\left(\f{1}{\ln \f{2}{1-|\sigma_w(z_n)|^2}}\right)^2<\infty,
\end{eqnarray*}
which means that $\sum_{n=1}^{\infty} (1-|z_n|^2)^t \delta_{z_n}$ is  a $t$-Carleson measure. From Lemma \ref{IB} again, $I\in C_\B(F(p, p-2, t))$ for any $t\in (s, 1)$. Hence
\begin{equation}\label{3-mid3}
I\in \bigcap_{s<t<1} C_\B(F(p, p-2, t)).
\end{equation}
Clearly,
\begin{equation}\label{3-mid4}
\bigcap_{s<t<1} C_\B(F(p, p-2, t))=\bigcap_{s<t<\infty} C_\B(F(p, p-2, t)).
\end{equation}
From (\ref{3-mid2}), (\ref{3-mid3}) and (\ref{3-mid4}),  we get
$$
C_\B(F(p, p-2, s)) \subsetneqq \bigcap_{s<t<\infty} C_\B(F(p, p-2, t)).
$$
The proof of Theorem \ref{3main} is finished.
\hfill{$\square$}

\end{document}